\documentclass[11pt,reqno]{amsart}
\usepackage[T1]{fontenc}
\usepackage[utf8]{inputenc}
\usepackage{amsmath,amssymb,amsthm,mathrsfs}
\usepackage[margin=1.15in]{geometry}
\usepackage[colorlinks,linkcolor=blue,citecolor=blue]{hyperref}

\makeatletter
\renewcommand\paragraph{\@startsection{paragraph}{4}%
  \z@{.5\linespacing\@plus.7\linespacing}{-.5em}%
  {\normalfont\bfseries}}
\makeatother

\theoremstyle{plain}
\newtheorem{theorem}{Theorem}
\newtheorem{proposition}[theorem]{Proposition}
\newtheorem{lemma}[theorem]{Lemma}
\newtheorem{corollary}[theorem]{Corollary}
\theoremstyle{definition}

\DeclareMathOperator{\sign}{sign}
\DeclareMathOperator{\type}{type}
\DeclareMathOperator{\Stab}{Stab}
\DeclareMathOperator{\prm}{prim}
\DeclareMathOperator{\midd}{mid}
\newcommand{\T}{\mathbb{T}}
\newcommand{\R}{\mathbb{R}}
\newcommand{\Z}{\mathbb{Z}}
\newcommand{\cD}{\mathscr{D}}
\newcommand{\cF}{\mathscr{F}}
\newcommand{\gradP}{\nabla_{\!P}}

\begin{document}

\title[A parity constraint on topological numbers]
{A parity constraint on the topological numbers\\ of self-complementary triply-periodic surfaces}

\author{Roberto De Leo} 
\address{Department of Mathematics, Howard University,\\ Washington, DC 20059, USA}
\email{roberto.deleo@howard.edu}

\begin{abstract}
    The stability zone of Novikov's problem of planar sections of triply periodic surfaces are labeled by indivisible triples of integers, which we call the soul of the zone. 
    Which triples may occur is not known. 
    We show that, if a Morse function $f$ on $\T^3$ admits a translation $\tau(x)=x+v$ with $f\circ\tau=-f$, so that the sets $\{f<0\}$ and $\{f>0\}$ are congruent, then in the generic genus-$3$ situation every topological number $L$ satisfies $\langle L,2v\rangle\equiv1\pmod2$.
    The mechanism is that the electron/hole type of a null-homologous closed trajectory is invariant under the period lattice and is reversed by $\tau$. 
    For $f=\cos2\pi x+\cos2\pi y+\cos2\pi z$ and the $\mu$-cube, this means that the sum of the components of  each soul is odd.
\end{abstract}

\maketitle

We work in the setting of Novikov's problem of planar sections of triply-periodic surfaces~\cite{Nov82,Zor84,Dyn93,Dyn97,Dyn99}; see~\cite{DeL19,DMN22} for surveys. 
Let $f\colon\T^3=\R^3/\Z^3\to\R$ be a Morse function, let $M=\{f=0\}$ and, for a direction $H$, denote by $\cF_H(M)$ the foliation of $M$ cut by the planes orthogonal to $H$, that is, the level sets of $h=\langle H,\cdot\rangle$ on $M$. 
When the foliation is \emph{integrable}, namely when the leaves not null-homologous in $\T^3$ fill components with genus $1$ and topological rank $2$, all such components carry one and the same indivisible class $L\in H_2(\T^3;\Z)$ up to sign~\cite{Dyn96,Dyn99}, that we call the \emph{soul}, and the set of directions with soul $L$ is the \emph{stability zone} $\cD_L$. 
The zones are open with piecewise smooth boundary and their union is dense~\cite{Dyn99}.

Which indivisible triples actually occur as souls is not known, and the question is not idle: Dynnikov's conjecture $\operatorname{Area}(\cD_L)\le C\|L\|^{-3}$ \cite{Dyn99} cannot hold for all $L$ simultaneously, since $\sum_L\|L\|^{-3}$ diverges, so the realised souls must be sparse. 
The purpose of this note is to record one exact restriction.

\paragraph{The case of the $\mu$-cube.}This restriction is visible in the only example whose souls are known in closed form. 
Let $\mathcal C$ denote the regular skew apeirohedron $\{4,6|4\}$, Coxeter's $\mu$-cube~\cite{Cox37}, realised in $\T^3$ as the zero set of
\begin{equation}\label{eq:mucube}
   \theta(x)=\midd\big(g(x),g(y),g(z)\big)-\tfrac12,
   \qquad g(t)=|2\{t\}-1|,
\end{equation}
where $\midd$ is the middle of three reals \cite{DD09}. 
The author, jointly with Dynnikov, determined $\cD(\mathcal C)$ completely \cite{DD09}: the three maximal zones have souls $(1,0,0),(0,1,0),(0,0,1)$, and the remaining zones are the middle triangles produced by the recursion
\begin{equation}\label{eq:LY}
   (a,b,c)\ \longmapsto\ (a,a{+}b,a{+}c),\quad (b,a{+}b,b{+}c),\quad (c,a{+}c,b{+}c),
\end{equation}
started from $\big((1,1,0),(1,0,1),(0,1,1)\big)$ and its sign variants, the middle triangle of $(a,b,c)$ having soul $\prm(a+b+c)$. 
The limit set of \eqref{eq:LY} is what we call the Novikov gasket, also called Levitt--Yoccoz gasket in~\cite{DeL15,DeL19} and Rauzy gasket in~\cite{AHS16,AS13}.

\begin{proposition}\label{prop:mucube}
Every soul of $\mathcal C$ has odd coordinate sum.
\end{proposition}

\begin{proof}
    Write $\sigma(u)=u_1+u_2+u_3$. Modulo $2$ the three initial vertices $(1,1,0),(1,0,1),(0,1,1)$ satisfy $a+b\equiv c$, $b+c\equiv a$, $c+a\equiv b$, so the three children in \eqref{eq:LY} reduce to $(a,c,b)$, $(b,c,a)$, $(c,b,a)$: the triple of residues is permuted and never changes. 
    Hence every vertex $u$ occurring in the recursion has $\sigma(u)$ even, and $a+b+c\equiv0\pmod 2$, say $a+b+c=2\ell$.

    Let $S=\sigma(a)+\sigma(b)+\sigma(c)$. 
    Initially $S=6$. Passing to the child that retains $a$,
\[
   S\ \longmapsto\ \sigma(a)+\sigma(a{+}b)+\sigma(a{+}c)=3\sigma(a)+\sigma(b)+\sigma(c)
   =S+2\sigma(a),
\]
    and $\sigma(a)$ is even, so $S$ is unchanged modulo $4$. 
    Thus $S\equiv2\pmod4$ throughout and $\sigma(\ell)=S/2$ is odd. 
    If $\ell=k\ell'$ with $\ell'$ indivisible then $k\mid\sigma(\ell)$, so $k$ is odd and $\sigma(\ell')$ is odd as well. The three maximal zones have $\sigma=1$.
\end{proof}

\paragraph{The main result.}
The proof above uses nothing but the arithmetic of \eqref{eq:LY}, and so explains nothing -- the explanation lies elsewhere. 
Since $g(t+\tfrac12)=1-g(t)$ for all $t$, and $u\mapsto1-u$ reverses order and therefore fixes the middle of a triple, \eqref{eq:mucube} gives
\begin{equation}\label{eq:anti}
   \theta\circ\tau=-\theta,\qquad \tau(x)=x+v,\quad v=(\tfrac12,\tfrac12,\tfrac12):
\end{equation}
the $\mu$-cube is \emph{self-complementary}, its interior being congruent to its exterior by a half-period translation, and $\sigma(\ell)=\langle \ell,2v\rangle$.
(Self-complementarity of a polyhedron is property SP1 of \cite{DeL06}, where it serves to identify the zones of the polyhedron with those of a function having it as a level set \cite[Theorem~7]{DeL06}.) 
We claim that this is the general reason.

\begin{theorem}\label{thm:main}
    Let $f\colon\T^3\to\R$ be Morse and suppose there is a translation $\tau(x)=x+v$ with $f\circ\tau=-f$.
    Then $w:=2v\in\Z^3$, and for every $H$ interior to a stability zone of $M=\{f=0\}$ whose decomposition is as in \textup{(S)} below, the soul $L$ satisfies
\[
   \langle L,w\rangle\ \text{is odd.}
\]
\end{theorem}

That $w\in\Z^3$ is automatic, since $f\circ\tau^2=f$ makes $2v$ a period of $f$. 
The hypothesis we use is the structure theorem in the simplest non-trivial case:
\begin{itemize}
\item[(S)] \emph{(Dynnikov \cite{Dyn96}, \cite[\S4.1]{Dyn99}; see
\cite[\S4.1.3]{DeL19})} 
For $H$ interior to a stability zone of a genus-$3$ surface of topological rank $3$, the foliation $\cF_H(M)$ is integrable and $M$ decomposes into exactly two components of genus $1$ and rank $2$ --- we call them \emph{carriers} --- separated by exactly two \emph{cylinders}, the components filled by leaves null-homologous in $\T^3$.
\end{itemize}

\begin{corollary}\label{cor:cos}
    For $\mathfrak c(x,y,z)=\cos2\pi x+\cos2\pi y+\cos2\pi z$ one has $\mathfrak c\circ\tau=-\mathfrak c$ with $v=(\tfrac12,\tfrac12,\tfrac12)$, so $w=(1,1,1)$ and every soul $(l,m,n)$ of $\mathfrak c_0=\{\mathfrak c=0\}$ has $l+m+n$ odd. 
\end{corollary}

Since $\mathfrak c_{-c}$ is a translate of $\mathfrak c_c$ one has $e_1(H)=-e_2(H)$ and $\cD(\mathfrak c)=\cD(\mathfrak c_0)$ \cite[\S4.1.3]{DeL19}, so Corollary~\ref{cor:cos} describes all of $\cD(\mathfrak c)$.
The stabililty zones of this surface were studied numerically in~\cite{DeL03}.

\paragraph{The case of polyhedra.}
Proposition~\ref{prop:mucube} is not a special case of Theorem~\ref{thm:main}: $\mathcal C$ is not a level set of a Morse function, and Dynnikov's argument for (S) in \cite[\S4.1]{Dyn99}, which concerns smooth surfaces of negative curvature whose height functions have four simple saddles, does not apply to it as stated.
Nevertheless, the results of Zorich and Dynnikov were extended by the author in \cite{DeL06} to embedded periodic polyhedra with convex plane faces and to directions $H$ orthogonal to no edge, here $H_1H_2H_3\neq0$; the critical points of $h|_{\mathcal C}$ are then vertices, a vertex being of index $1-s/2$ when $h$ changes sign $s$ times around it \cite[Definitions~2 and~3]{DeL06}.
$\mathcal C$ has $V=8$ vertices, each of degree $6$, hence $E=24$ and $F=12$, so $\chi=-4$; equivalently the angle $6\cdot\tfrac\pi2$ at each vertex gives excess $\pi$, and Gauss--Bonnet gives the same.
Thus, $\mathcal C$ has genus $3$.
By \eqref{eq:mucube} the six edges at a vertex point, in cyclic order, along $\epsilon_1e_1,-\epsilon_2e_2,\epsilon_3e_3,-\epsilon_1e_1,\epsilon_2e_2,-\epsilon_3e_3$, where $\epsilon_i=1$ or $-1$ according as the $i$-th coordinate of the vertex is $\frac14$ or $\frac34$.
Hence $h$ changes sign six times around the two vertices with $(\epsilon_i\sign H_i)_i=\pm(1,1,1)$ and twice around the other six: for every such $H$ the only critical points are two monkey saddles of index $-2$, exchanged by $\tau$, and not four simple saddles.

The proof of Theorem~\ref{thm:main} uses $f$ only through the two sides of $M$, which $\tau$ exchanges (a compact leaf $\gamma$ is electron-like exactly when $R_\gamma$ lies in $\{f<0\}$ near $\gamma$; see the on electron-like and hole-like trajectories below), and through the decomposition in (S).
It therefore applies verbatim to any self-complementary periodic polyhedron at every $H$ at which that decomposition holds.
We do not verify this for $\mathcal C$, and we record Proposition~\ref{prop:mucube} not for its generality but because in that one case the parity can be read off the combinatorics of \eqref{eq:LY} in a few lines, with no reference to the geometry of the foliation.
What the general theorem adds is the reason.

\paragraph{Electron-like and hole-like trajectories.}\label{sec:type}
Fix $H$ with $|H|=1$, let $\tilde M\subset\R^3$ be the lift of $M$, and let $\gradP f=\nabla f-H\langle H,\nabla f\rangle$ be the gradient of the restriction of $f$ to the planes orthogonal to $H$. 
A compact leaf $\gamma$ of $\cF_H(\tilde M)$ bounds a unique compact disc $R_\gamma$ in its plane. 
Following Dynnikov~\cite[\S2.3, Definition~3]{Dyn99}, we say that $\gamma$ is \emph{electron-like} if $\gradP f$ is an outward normal to $R_\gamma$ at every point of $\gamma$, and \emph{hole-like} if it is an inward normal at every point of $\gamma$; we write $\type(\gamma)=+1$ or $-1$ accordingly. 
Thus, $\gamma$ is electron-like exactly when it encircles a region of $\{f<0\}$, which is the classical distinction between electron and hole orbits~\cite{LAK56,NM98}. 
We use two properties of the type that are already in \cite{Dyn99}, where it is defined for closed trajectories on $M\subset\T^3$ and assigned to whole cylinders \cite[\S\S2.3, 4.1]{Dyn99}:
\begin{itemize}
    \item[(D1)] every compact leaf is electron-like or hole-like, and the type is constant on each cylinder (since $\gradP f$ is a nonzero normal to $\gamma$ inside its plane, varying continuously over the connected cylinder);
    \item[(D2)] $\type(\gamma+z)=\type(\gamma)$ for $z\in\Z^3$ (since $\nabla f$ is $\Z^3$-periodic).
\end{itemize}

\medskip
To these, we add the following elementary observation.

\begin{lemma}\label{lem:invariance}
    If $f\circ\tau=-f$, then $\type(\tau\gamma)=-\type(\gamma)$ for every compact leaf $\gamma$.
\end{lemma}

\begin{proof}
    The translation $\tau$ carries the plane of $\gamma$ to a parallel plane, $R_\gamma$ to $R_{\tau\gamma}=R_\gamma+v$, and outward normals to outward normals, while differentiating $f(x+v)=-f(x)$ gives $\nabla f(x+v)=-\nabla f(x)$, hence $\gradP f(x+v)=-\gradP f(x)$.
\end{proof}

Lemma~\ref{lem:invariance} is the whole mechanism: by (S) there are exactly two
cylinders, and $\tau$ permutes them while reversing a $\pm1$ invariant, so it cannot fix either one. 
The rest of the proof turns this into a statement about $L$.

\paragraph{The chain of components and its shift homomorphism.}

Let $\Lambda=L^\perp\cap\Z^3$, of rank $2$ and covolume $\|L\|$. The lift $\tilde M$ is connected. 
Each of its carriers has stabiliser $\Lambda$ in $\Z^3$, hence is a $\Lambda$-cover of a compact surface; the preimage of a cylinder is a union of copies of
it indexed by $\Z^3$, because the core of a cylinder is null-homologous in $\T^3$. 
Call the $\Lambda$-orbit of lifted cylinders lying between two consecutive carriers a \emph{layer}.

\begin{lemma}\label{lem:chain}
    The carriers and layers of $\tilde M$ form a bi-infinite path, so they may be indexed by $\Z$ with the $k$-th layer adjacent exactly to the $k$-th and $(k{+}1)$-st carrier.
    Every translation $g$ preserving $\tilde M$ acts on the index by a shift $\psi(g)\in\Z$, and $\psi$ is a homomorphism on $G=\langle\Z^3,\tau\rangle$.
\end{lemma}

\begin{proof}
    In $\T^3$ the incidence graph has two carrier and two cylinder vertices, each of degree $2$ by (S), so it is either a $4$-cycle or a pair of $2$-cycles; the latter would disconnect $\tilde M$. 
    Its $\Z^3$-cover is therefore a bi-infinite path.

    A translation preserving $\tilde M$ and the plane foliation permutes the components and induces an automorphism of the path, so a shift or a reflection. For a carrier $T_k$ the quotient $T_k/\Lambda$ is compact, so $\Pi:=\langle L,\cdot\rangle$ is bounded on $T_k$; as $\tilde M$ is properly embedded, $\inf_{T_k}\Pi\to\pm\infty$ with $k\to\pm\infty$, so the order of the path agrees at infinity with $\Pi$. 
    A translation changes $\Pi$ by a constant and hence preserves that order. 
    (A reflection would force $\psi(g^2)=0$, so $\langle L,w\rangle=0$ and $\tau$ would preserve $\Pi$, which forbids exchanging carriers on opposite sides.) 
    Shifts compose, so $\psi$ is a homomorphism.
\end{proof}

\begin{lemma}\label{lem:psi}
    $\ker\psi=\Lambda$, and $\psi(z)=2\langle L,z\rangle$ for $z\in\Z^3$ up to a global sign.
\end{lemma}

\begin{proof}
    $\psi(g)=0$ means that $g$ fixes every carrier, so $\ker\psi=\Stab(T_0)$, which is the image of $\pi_1(K)\to\pi_1(\T^3)=\Z^3$ for the corresponding carrier $K\subset M$. The boundary circles of $K$ are null-homotopic in $\T^3$ \cite{Zor84}, so that image is unchanged when they are capped off, and for the resulting torus of indivisible class $L$ it equals $\Lambda$. 
    Hence $\psi$ factors injectively through $\Z^3/\Lambda\cong\Z$, generated by the class of any $z_0$ with $\langle L,z_0\rangle=1$. 
    The translation $k\mapsto k+\psi(z_0)$ of $\Z$ has $|\psi(z_0)|$ orbits, and these are the layers of $M$, of which there are two by (S).
\end{proof}

No capping of carriers inside $\R^3$ is used: Lemma~\ref{lem:psi} caps only to identify a $\pi_1$-image, which is a statement about $M$ itself.

\paragraph{Proof of Theorem~\ref{thm:main}}

By Lemma~\ref{lem:psi} there is $z\in\Z^3$ with $\psi(z)=2$; it carries the layer of index $k$ to that of index $k+2$ by Lemma~\ref{lem:chain} and preserves the type by (D2). 
Writing $t(k)$ for the type of the $k$-th layer, well defined by (D1) and (D2),
\[
   t(k+2)=t(k)\qquad\text{for all }k,
\]
so $t$ depends only on $k$ modulo $2$. 
Put $s=\psi(\tau)$. 
By Lemmas~\ref{lem:chain} and \ref{lem:invariance},
\[
   t(k+s)=-t(k)\qquad\text{for all }k,
\]
and if $s$ were even the left-hand side would equal $t(k)$, forcing $t(k)=0$. 
Hence $s$ is odd. 
Finally $\tau^2$ is the translation by $w$, so $2s=\psi(\tau^2)=\psi(w)=2\langle L,w\rangle$ by Lemma~\ref{lem:psi}, that is $s=\langle L,w\rangle$. \qed

\paragraph{Complements and questions}

If (S) is replaced by ``exactly $\nu$ carriers and $\nu$ cylinders'', the same argument gives $\psi(z)=\nu\langle L,z\rangle$ and $s=\tfrac\nu2\langle L,w\rangle$, and requires a map $t\colon\Z/\nu\to\{\pm1\}$ with $t(k+s)=-t(k)$ to exist, that is, that $\nu/\gcd(\nu,s)$ be even. 
Here $\nu$ is always even, since the carriers, capped off, represent $\pm L$ and their classes add up to $[M]=0$ \cite[\S3.1, Proposition~3]{Dyn99}; so $s$ is an integer. 
For $\nu=2$, this is Theorem~\ref{thm:main}.

We close with three questions.
\begin{enumerate}
    \item Does the conclusion persist beyond (S), for surfaces of higher genus?
    The remark above reduces this to the arithmetic of $\nu$ and $s$ but, beyond the parity of $\nu$, we do not know which $\nu$ occur.
    \item Is there a converse: for which $L$ with $\langle L,w\rangle$ odd is $\cD_L\neq\emptyset$ for some self-complementary $f$?
    \item Theorem~\ref{thm:main} halves the set of candidate souls, whereas the divergence of $\sum_L\|L\|^{-3}$ shows that the realised souls must be considerably sparser if Dynnikov's area conjecture holds. 
    Can the parity constraint be strengthened in that direction?
\end{enumerate}

Numerical computations support a positive answer to the first question.
The function 
$$
f(x,y,z)=\cos2\pi x+\cos2\pi y+\cos2\pi z+1.3\cos2\pi(2x+y)
$$ 
satisfies $f\circ\tau=-f$ with $v=(\tfrac12,\tfrac12,\tfrac12)$, so again $w=(1,1,1)$, but its zero level has genus $5$, so (S) does not apply.
We computed the souls of a piecewise-linear approximation of $\{f=0\}$, constructed to be exactly invariant under $\tau$, at $154{,}897$ directions in the square $0\le H_1/H_3,\,H_2/H_3\le1$, obtaining $14{,}961$ distinct souls and every one of them has $\langle L,(1,1,1)\rangle$ odd.

\section*{Acknowledgments}
This material is based upon work supported in part by the National Science Foundation under Grant Number DMS-2308225.

\section*{Tools use}
The results in this article originated from discussions on the topic with Claude, an AI model by Anthropic.
The numerical results were obtained using a recent version of the author's library NTC updated to genus 5 by Claude.

\end{document}